\documentclass[reqno]{amsart}

\usepackage{amsmath,amssymb,enumerate, amsfonts,mathtools}
\usepackage{epsfig,amsthm}
\usepackage[dvipsnames]{xcolor}
\usepackage{soul}

\newtheorem*{theorem*}{Theorem}
\newtheorem*{corollary*}{Corollary}
\newtheorem{theorem}{Theorem}

\newtheorem*{conjecture*}{Conjecture}

\newtheorem{proposition}[theorem]{Proposition}
\newtheorem{lemma}[theorem]{Lemma}
\newtheorem{corollary}[theorem]{Corollary}

 \theoremstyle{remark}
\newtheorem{remark}[theorem]{Remark}

\newcounter{jbStepCounter}

\newcommand{\occult}[1]{}

\newcommand\SY[1]{{\noindent\color{red}  #1}}

\usepackage{soul}

\newcommand\eps{\varepsilon}

\newcommand\cR{\mathcal R}

\newcommand\RR{{\mathbb R}}

\newcommand\Card{\operatorname{Card}}

\begin{document}
\author{J\'er\^ome Buzzi, Sylvain Crovisier and Omri Sarig}
\thanks{J.B. was partially supported by the ANR project 16-CE40-0013 \emph{ISDEEC}, S.C. by  the  ERC  project 692925 \emph{NUHGD} and O.S. by the ISF grant  1149/18.}
\title[Existence of SRB measures]{Another proof of Burguet's existence theorem for\\    SRB measures of $C^\infty$ surface diffeomorphisms}
\begin{abstract}
Recently,
Burguet  proved  a strong form of Viana's conjecture  on physical measures \cite{V}, in the special case of  $C^\infty$ surface diffeomorphisms. We give another proof, based on  our analysis of entropy and Lyapunov exponents in \cite{BCS}.
\end{abstract}
\maketitle

\newcommand\Basin{\operatorname{Basin}}
\newcommand\wto{\stackrel*\longrightarrow}

\mbox{}
Given a smooth diffeomorphism $f$ of a closed  Riemannian manifold $M$, one would like to build \emph{physical measures}:
 Borel probability measures $\mu$ whose \emph{ergodic basins}
 $$
     \Basin(\mu):=\left\{x\in M: \lim_{n\to\infty}\frac1n\sum_{k=0}^{n-1} \delta_{f^k(x)}\overset{weak-*}{\xrightarrow{\hspace*{1cm}}} \mu\right\}\ \ \ ( \delta_y:=\text{point mass at $y$})
  $$
have positive Riemannian volume. Some diffeomorphisms do not have physical measures \cite{takens}, and a fundamental problem in smooth dynamics is to determine when they exist.
Following positive results on uniformly hyperbolic systems \cite{Ruelle},\cite{Sinai},\cite{Bowen-Ruelle}, for interval maps \cite{Keller}, for some partially or non-uniformly hyperbolic systems \cite{Young-Tower,ABV,BV}, or assuming regularity in the spirit of Oseledets \cite{Tsujii}, Viana~\cite{V} conjectured the following:\footnote{In \cite{V}, the name \emph{SRB measures} is used for what we call \emph{physical measures}.}

\medskip
\noindent
{\bf Viana's Conjecture:} {\em
Suppose a smooth map $f$ has only non-zero Lyapunov exponents at Lebesgue almost every point. Then $f$ has a physical measure.}

\medskip
\noindent
This has led to many works, with some notable recent advances assuming some recurrence properties (in particular, \cite{O,CLP} and the references therein).
\medbreak

Burguet has recently  proved  Viana's conjecture in the special case of $C^\infty$ surface diffeomorphisms. We need some notation to state his result. Suppose $\dim M=2$ and $f\in C^\infty$.
The \emph{upper Lyapunov exponent at $x\in M$} is $\lambda^+(x):=\limsup \tfrac 1 n \log \|Df^n(x)\|$.
The upper Lyapunov exponent of an  $f$-invariant but possibly non-ergodic probability measure $\mu$  is  $\lambda^+(\mu)=\int \lambda^+(x) d\mu$.
An ergodic measure $\mu$ is a \emph{Sina\"\i-Ruelle-Bowen (SRB) measure if
its entropy $h(f,\mu)$ is positive and equal to $\lambda^+(\mu)$.\footnote{The definition of the SRB property varies in the literature. In higher dimension, one may require as in~\cite{Y} that $h(f,\mu)$ coincides with the sum of all the positive Lyapunov exponents of $\mu$ and that it is nonzero.
 But in dimension bigger than two, this allows zero exponents, and it does not imply physicality.}}
 The Ledrappier-Young Theorem says that this  is equivalent to the existence of a system of absolutely continuous conditional measures on local unstable manifolds~\cite{LY}.  In dimension two,  every SRB measure is physical ~\cite{L}.
The result we are interested in is:

\begin{theorem*}[Burguet]
Let $f$ be a $C^\infty$ diffeomorphism on a closed surface such that, on a set of points $x$ with positive Riemannian area,
$\limsup_{n\to\infty} \tfrac 1 n \log \|Df^n(x)\|>0$.
Then $f$ admits an SRB measure.
\end{theorem*}

\noindent
This  implies Viana's conjecture, in the special case of smooth surface diffeomorphisms.

\medbreak

The goal of this note is to explain how to prove this theorem using some tools we developed in~\cite{BCS}, for the analysis of the continuity properties of entropy and exponents.
Burguet used a different method, which allowed him to also describe the basin of the SRB measures.

\section{Proof of the theorem}
Let $f:M\to M$ be a $C^\infty$ diffeomorphism on a closed smooth surface $M$. For most of the proof we fix $\gamma,\eta>0$ small, $r\geq 2$ a large integer, and work in $C^r$-regularity.
(We will also need a condition $\gamma<\gamma_0(r, f,\eta)$.) At the end we let $\gamma,\eta\to 0$ and $r\to \infty$.

\subsection{Projective dynamics}\label{ss.lifted}
The {\em projective tangent bundle of $M$} is the  manifold $\widehat{M}$  of pairs $\xi=(x,E)$ where $x\in M$
and $E$ is a one-dimensional linear subspace of $T_xM$,  with the Riemannian structure inherited from $TM$.  Let $\pi\colon \widehat M\to M$ be the natural projection.

The diffeomorphism $f$ induces the diffeomorphism 
$\widehat f(x,E)=(f(x),Df(E))$ on $\widehat{M}$.
Let $\varphi\colon \widehat M\to \RR$ be  the continuous function
$$\varphi(x,E)=\log \|Df|_E\|.$$
\begin{lemma}\label{l.lambda-plus-formula}
If $\widehat \mu$ is an $\widehat{f}$-ergodic invariant measure, $\int \varphi d\widehat \mu>0$, and  $\mu:=\pi_*\widehat \mu$ is not supported on a source, then
$\lambda^+(\mu)=\int \varphi d\widehat \mu$.
\end{lemma}
\noindent
 The (standard) proof can be found e.g. in \cite[Lemma 3.3]{BCS}.

\newcommand\hsigma{\widehat\sigma}
\subsection{Reduction to a curve}\label{s.Reduction-to-curve}
A curve $\sigma\colon [0,1]\to M$ is
\emph{regular} if it is $C^r$, injective, and if the derivative $\sigma'$ never vanishes.
It lifts canonically as a curve $\widehat \sigma\colon [0,1]\to \widehat M$ defined by
$$\widehat \sigma (s):=(\sigma(s),\;\RR.\sigma'(s)).$$
 Fix once and for all a finite atlas for $\widehat{M}$
 so we can define the $C^r$ size as follows.
Suppose $\eps,\widehat\eps\in (0,1)$. We say that $\widehat\sigma$ has \emph{$C^r$ size less than $(\varepsilon,\widehat\varepsilon)$}
if, using the charts from the atlas, the norms of the $k^{\text{th}}$-order derivatives of $\sigma$ with $1\leq k\leq r$
are bounded by $\varepsilon$, and $k^{\text{th}}$-order derivatives of $\widehat \varepsilon$ with  $1\leq k\leq r-1$ are bounded by $\widehat \varepsilon$.

\begin{lemma}\label{l.curve}
There exists a regular  $C^\infty$ curve $\sigma$ such that, for a set of parameters $s\in[0,1]$ with positive Lebesgue measure,
$\limsup \tfrac 1 n \log \|Df^n|_\sigma(\sigma(s))\|>0$ and $\sigma(s)$ is not  $f$-periodic.
\end{lemma}
\begin{proof} The condition
$\limsup \tfrac 1 n \log \|Df^n(x)\|>0$ implies that $\limsup \tfrac 1 n \log \|Df^n|_E(x)\|>0$ for every one-dimensional linear subspace $E$ of $T_x M$, except  maybe one. Indeed let us consider a sequence $n_k\to +\infty$ which realizes the supremum of $(\tfrac 1 n \log \|Df^n(x)\|)$. Let $F_{n_k}$
be a one-dimensional space of maximal contraction for $Df^{n_k}(x)$. One can assume that $(F_{n_k})$ converges to a space $F$.
If $E\neq F$, then for $k$ large $\|Df^{n_k}|_E(x)\|\geq \frac12\sin\angle(E,F) \| Df^{n_k}(x)\|$.

Hence there exists a compact set $K$ with positive two-dimensional Lebesgue measure and a smooth tangent line field $L$ on an open neighborhood of $K$
such that the expansion $\limsup \tfrac 1 n \log \|Df^n_x|_E\|>0$ holds for any $x\in K$ and any direction $E\subset T_xM\setminus L_x$.
By Fubini's theorem, there exists a smooth curve $\sigma$ transverse to the tangent line field $L$ which intersects $K$ along a set with
positive one-dimensional Lebesgue measure.
The expansion implies that the periodic points in $K\cap \sigma[0,1]$ are isolated, whence  at most countable many.
\end{proof}

Let $\overline \lambda>0$ be the essential supremum of the function $s\mapsto \limsup \tfrac 1 n \log \|Df^n|_\sigma(\sigma(s))\|$. (Here and throughout, $Df^n|_{\sigma}(p):=(Df^n)_p|_{T_p\sigma}$ for $p$ on $\sigma$.)
\begin{lemma}\label{l.subset}
Given $0<\lambda^{min}<\overline \lambda<\lambda^{max}$ and $0\!<\!\rho\!<\!1$, there is $n$ arbitrarily large such that
$$\begin{aligned}
T_n:=\{s\in[0,1]:\;& \lambda^{min}\leq\tfrac 1 n \log \|Df^n|_\sigma(\sigma(s))\|\leq \lambda^{max},\\
& \text{ $\sigma(s)$ is not $f$-periodic} \text{ and }\; \forall 0\leq i\leq n,\;\; \|Df^i|_\sigma(f^{n-i}\sigma(s))\|\geq 1\}
\end{aligned}$$
has  Lebesgue measure larger than $\rho^n$.
\end{lemma}
\begin{proof}
By Lemma~\ref{l.curve}, if $\varepsilon>0$ is small, then
for $s$ in a set with positive Lebesgue measure, $\sigma(s)$ is not periodic and $\limsup \tfrac 1 n \log \|Df^n|_\sigma(\sigma(s))\|\in [\lambda^{min}+\varepsilon, \lambda^{max}-\varepsilon]$.
By Pliss's lemma (see~\cite{M}) there exists an arbitrarily large  integer $n$ such that
$ \|Df^k|_\sigma(f^{n-k}\sigma(s))\|\geq \exp(\lambda^{min}\cdot k)$ for all $0\leq k\leq n$.
Note also that the inequality $\|Df^n|_\sigma(\sigma(s))\|\leq \exp(\lambda^{max}\cdot n)$ is satisfied for all $n$ large.
Hence $s$ belongs to infinitely many sets $T_n$.
By the Borel-Cantelli lemma, the Lebesgue measure of $T_n$ cannot be  smaller than $\rho^n$ for all large $n$.
\end{proof}

We continue the proof by choosing some sequences $\lambda_k^{min},\lambda^{max}_k\to \overline \lambda$, $\rho_k\to 1$. By Lemma~\ref{l.subset},
there exist $n_k\to \infty$ such that the sets $T_{n_k}$ have  Lebesgue measure larger than $\rho_k^{n_k}$.

\subsection{Empirical measures and  neutral decompositions}\label{s.Neutral-decomposition}
Let $\xi\in \widehat M$ be a point. The $k^\text{th}$ \emph{empirical measure}
and \emph{Birkhoff sum} of $\varphi$ at $\xi$ are
$$p_k(\xi):=\frac 1 {n_k} \sum _{0\leq i\leq n_k-1} \delta_{\widehat f^i(\xi)},
\qquad S_{n_k}\varphi(\xi):=\varphi(\xi)+\varphi(\widehat f(\xi))+\dots\varphi(\widehat f^{n_k-1}(\xi)).$$
So $\tfrac{1}{n_k}S_{n_k}\varphi(\xi)=\int \varphi dp_k(\xi)$.
Note that by definition of $\varphi$ and $T_{n_k}$,
$$\forall \xi\in \widehat \sigma(T_{n_k}),\forall 0\leq i\leq n_k,\;\; S_{i}\varphi(\widehat f^{n_k-i}(\xi))\geq 0.$$

\noindent
\paragraph{\it Neutral orbit segments.}
Suppose  $\alpha>0$ and $L\geq 1$.  An orbit segment $(\zeta,\widehat f(\zeta),\dots,\widehat f^{\ell-1}(\zeta))$
is called  \emph{$\alpha$-neutral}, if
$S_i\varphi(\zeta)\leq \alpha i$ for all $0< i \leq  \ell$.
An $\alpha$-neutral orbit segment of length $\ell\geq L$ is called  \emph{$(\alpha,L)$-neutral}.

The union of two intersecting $\alpha$-neutral segments  is still a neutral segment.
Consequently, the union of all $\alpha$-neutral sub-segments of a given finite orbit segment $(\xi,\widehat f(\xi),\dots,\widehat f^{n_k-1}(\xi))$ is a pairwise disjoint union of maximal $\alpha$-neutral sub-segments.
\smallskip

\noindent
\paragraph{\it Neutral part of an empirical measure.} 
This is the measure
$$
p_k^{\alpha,L}(\xi):=\frac{1}{n_k}\sum_{0\leq i\leq n_k-1} 1_{N_{n_k}^{\alpha,L}(\xi)}(\widehat{f}^i(\xi)) \cdot\delta_{\widehat{f}^i(\xi)},
$$
where $N^{\alpha,L}_{n_k}(\xi)$ is the union of $(\alpha,L)$-neutral sub-segments of $(\xi,\ldots,\widehat{f}^{n_k-1}(\xi))$.
%
%
\bigskip

The following  proposition is a version of \cite[Proposition 6.2]{BCS}. There we were given a sequence of ergodic measures with nonnegative averages,
and we started by constructing empirical measures along generic points. Here we skip the first step and treat the empirical measures as given:

\newcommand\hM{\widehat M}
\newcommand\hf{\widehat f}

\begin{proposition}\label{p.decomposition}
Let $(\xi_k)$ be non-periodic points in $\hM$ satisfying $S_{i}\varphi(\widehat f^{n_k-i}(\xi_k))\geq 0$
for each $0\leq i\leq n_k$.
Up to extracting a subsequence, there exist $\beta\in [0,1]$ and two $\widehat f$-invariant probability measures
$\widehat \mu_0,\widehat \mu_1$ as follows:
\begin{enumerate}[(a)]
\item The sequence of measures $p_k(\xi_k)$ converges to $\widehat \mu:=(1-\beta) \widehat \mu_0 + \beta \widehat \mu_1$  weak-$\ast$ on $\widehat{M}$.
\item For any pair of neighborhoods $\widehat V_0,\widehat V_1$ of $(1-\beta) \widehat \mu_0$  and  $\beta \widehat \mu_1$, if $\alpha>0$ is small and $L$ is large, then for all $k$ large enough, one has
$p_k^{\alpha,L}(\xi_k)\in \widehat V_0$ and $[p_k(\xi_k)-p_k^{\alpha,L}(\xi_k)]\in \widehat V_1$.
\item $(1-\beta)\int \varphi d\widehat\mu_0=0$.
\item For $(\beta\widehat\mu_1)$-almost every point $\xi$, the limit of $\tfrac 1 n S_n\varphi(\xi)$ is positive.
\end{enumerate}
\end{proposition}
\begin{remark}
The proposition holds for any homeomorphism $\widehat f$ of a compact space $\widehat M$,
for any continuous observable $\varphi$, and for any sequence $n_k\to+\infty$.
\end{remark}
\begin{proof}
Up to taking a subsequence, $(p_k(\xi_k))$ converges to an invariant measure $\widehat \mu$, and,
by a diagonal argument and using the monotonicity with respect to $\alpha$, $(p_k^{\alpha,L}(\xi_k))$ converges for each $\alpha>0$ and $L\geq 1$
to some measure $m_{\alpha,L}$  (see \cite[Claim 6.4]{BCS}).
Again by monotonicity, $m_{\alpha,L}$ converges as $\alpha\to 0$, $L\to +\infty$ to
the measure $m_0:=\inf_{\alpha,L}m_{\alpha,L}$.
Note that $\widehat f_*m_{\alpha,L}-m_{\alpha,L}\to 0$ as $L\to +\infty$. Hence $m_0$ and $m_1:=\widehat \mu-m_0$ are  $\widehat f$-invariant.
We set $\beta:=m_1(\widehat M)$ and $m_0=(1-\beta)\widehat \mu_0$, $m_1=\beta\widehat \mu_1$.
The items (a) and (b) follow.

Let $(\zeta,\widehat f(\zeta),\dots,\widehat f^{\ell-1}(\zeta))$ be a maximal $(\alpha,L)$-neutral segment in
$(\xi,\widehat f(\xi),\dots,\widehat f^{n_k-1}(\xi))$. We have $S_\ell\varphi(\zeta)\leq \alpha\cdot\ell$.
If $\widehat f^{\ell-1}(\zeta)\neq \widehat f^{n_k-1}(\xi)$, then
$\alpha\cdot\ell-\|\varphi\|_\infty<S_\ell\varphi(\zeta)$ by maximality. Otherwise $S_\ell\varphi(\zeta)=S_\ell\varphi(\widehat f^{n_k-\ell}(\xi_k))\geq0$ by our assumption on the points $\xi_k$. Thus this gives $-\|\varphi\|_\infty/L\le p^{\alpha,L}_k(\xi_k)(\varphi)\leq\alpha$ and taking the limit in $\alpha$ and $L$, one gets item (c).

Item (d) is proved by contradiction: we assume $\lim \tfrac 1 n S_n\varphi(\xi)\leq 0$ on a set of points $\xi$ with $\beta\widehat \mu_1$-measure $\chi>0$. We also fix $\alpha>0$ small, $L\geq 1$ large so that the mass $|m_{\alpha,L}-m_0|=(m_{\alpha,L}-m_0)(\hM)$ is less than $\chi/10$.
By Pliss' lemma, there exists $\ell\geq L$ such that the compact set $K_{\alpha/2,\ell}$ of points $\xi$ which belong to some
$\alpha/2$-neutral segment of length $\ell$ has $\beta\widehat \mu_1$-measure larger than $\chi/2$.
The set $K_{\alpha,\ell}$ of points which belong to some $\alpha$-neutral segment of length $\ell$ is a neighborhood of $K_{\alpha/2,\ell}$.
With item (b), one deduces that for some $\alpha'<\alpha$, $L'>L$ and for any $k$ large, $(p_k(\xi_k)-p_k^{\alpha',L'}(\xi_k))$ gives a mass larger than $\chi/3$ to
$K_{\alpha,\ell}$. By our choice of $\alpha,L$, the measure $(p_k^{\alpha,L}(\xi_k)-p_k^{\alpha',L'}(\xi_k))$ has mass less than $\chi/6$, for $k$ large.
Hence $(p_k(\xi_k)-p_k^{\alpha,L}(\xi_k))$ gives a mass larger than $\chi/6$ to $K_{\alpha,\ell}$.
Since $\ell\geq L$ this implies that $(\xi_k,\widehat f(\xi_k),\dots,\widehat f^{n_k-1}(\xi_k))$
contains some $(\alpha,L)$-neutral segment which is not included in the support of $p_k^{\alpha,L}(\xi_k)$.
This contradicts the definition of $p_k^{\alpha,L}(\xi_k)$.
\end{proof}

\subsection{A sequence of subsets $T_{n_k}'$}\label{s.subsets}
%
We endow the space of measures with a metric compatible with the weak-$\ast$ topology, and denote the Lebesgue measure of $T\subset [0,1]$ by $|T|$.

\begin{lemma}\label{l.sequence}
 Up to passing to a subsequence, there exist $\alpha_k\to 0$, $L_k\to\infty$, $\delta_k\to 0$
and subsets $T'_{n_k}\subset T_{n_k}$ with $|T'_{n_k}|> \rho_k^{2n_k}$
such that 
for any $\xi,\xi'\in \widehat \sigma(T'_{n_k})$, 
the empirical measures satisfy 
$d(p_k(\xi),p_k(\xi'))<\delta_k$
and $d(p^{\alpha,L}_k(\xi),p^{\alpha,L}_k(\xi'))<\delta_k$
for all $\alpha=\alpha_j$, $L=L_j$ with $j\leq k$.
\end{lemma}
\begin{proof}
It is enough to show that, given
a finite set of pairs $(\alpha,L)$ and $\delta>0$,
for $k$ large, there exists $T'_{n_k}\subset T_{n_k}$ with $|T'_{n_k}|\geq \rho_k^{2n_k}$
such that for each of these pairs $(\alpha,L)$ and any $\xi,\xi'\in \widehat \sigma(T'_{n_k})$, one has $d(p_k(\xi),p_k(\xi'))<\delta$
and $d(p^{\alpha,L}_k(\xi),p^{\alpha,L}_k(\xi'))<\delta$.

The set of measures $m$ with mass $m(\widehat M)\leq 1$ can be covered by some finite number $K$
of balls of radius $\delta/2$.
Hence, for each $k$ there exists a set $T'\subset T_{n_k}$
with $|T'|\geq |T_{n_k}|/K^2$ such that for any $\xi,\xi'\in \widehat \sigma(T)$, each pair of measures
$\{p_k(\xi),p_k(\xi')\}$ and $\{p^{\alpha,L}_k(\xi),p^{\alpha,L}_k(\xi')\}$
lie in some  ball of the covering.
Finally, for $k$ large enough, $|T_{n_k}|/K^2\geq \rho_k^{2n_k}$.
\end{proof}

From now on, we consider a sequence of integers $k\to\infty$ as in the previous Lemma.
We will choose points $\xi_k\in T'_{n_k}$ and obtain the limits  $\widehat \mu:=(1-\beta) \widehat \mu_0 + \beta \widehat \mu_1$ from Proposition~\ref{p.decomposition}.
By Lemma~\ref{l.sequence},  $p_k(\xi_k')\to \widehat \mu$ and $p_k^{\alpha,L}(\xi_k')\to (1-\beta)\widehat \mu_0$ for any 
choice of the 
sequence of $\xi'_k\in T'_{n_k}$.

\begin{corollary}\label{c.decomposition}
One can choose the points $\xi_k\in T'_{n_k}$ so that
the measure $\mu_1:=\pi_*(\widehat \mu_1)$  gives zero measure to all repelling hyperbolic periodic orbits, and 
$\beta \lambda^+(\mu_1)=\overline \lambda>0$. Thus $\beta\neq 0$.
\end{corollary}

\begin{proof}
By construction $\int \varphi d p_k(\xi_k) \in [\lambda^{min}_k,\lambda^{max}_k]$. Since $(1-\beta)\int\varphi d\widehat \mu_0=0$ by item (c) of Proposition~\ref{p.decomposition}, at the limit one gets
$\overline \lambda=\int \varphi d \widehat \mu=\beta \int\varphi d\widehat \mu_1$. In particular $\beta>0$.

The ergodic components $\widehat \mu'_1$ of $\widehat \mu_1$ satisfy $\int \varphi d\widehat \mu'_1>0$ by Proposition~\ref{p.decomposition} (d).
We claim that their projections $\mu'_1=\pi_*\widehat \mu'_1$ are not supported on repelling  hyperbolic periodic orbits.

 Otherwise there are $\chi>0$ and an  arbitrarily small open neighborhood $U$ of a repelling hyperbolic periodic orbit such that  $\widehat{\mu}(\partial U)=0$, and
$p_k(\xi_k)(U)>\chi$. 
By Lemma \ref{l.subset}, 
$S_k:=\{s\in [0,1],\; p_k(\sigma(s))(U)>\chi/2\}$ contains $T_{n_k}'$ for every $k$ large. But  $|T_{n_k}'|\geq \rho_k^{2k}$ with $\rho_k\to 1$, whereas if $U$ is a sufficiently small, then $|S_k|\to 0$ exponentially, a contradiction.

So $\widehat{\mu}_1$ does not charge sources. By Lemma \ref{l.lambda-plus-formula}, $\lambda^+(\mu'_1)=\int\varphi d\widehat \mu'_1$. So $\lambda^+(\mu_1)=\int\varphi d\widehat \mu_1$.
\end{proof}

\subsection{Reparametrizations: tools}
Let $N\geq 1$ and $\varepsilon,\widehat \varepsilon\in(0,1)$ (to be specified later).
\begin{enumerate}[]
\item A \emph{reparametrization of $\sigma$} is a non-constant affine map $\psi\colon [0,1]\to [0,1]$.

\item A \emph{family of reparametrizations over {a subset} $T\subset[0,1]$} is a collection $\mathcal R$ of reparametrizations such that
$T\subset \bigcup_{\psi\in\mathcal R} \psi([0,1])$.

\item A reparametrization $\psi$ is {\em $(N,\eps,\widehat \eps)$-admissible up to time $n$},  if
there exist $n_0,\dots,n_\ell$ s.t.
\begin{itemize}
\item $n_0=0$, $n_\ell=n$, and $1\leq n_{j}-n_{j-1}\leq N$ for each $1\leq j\leq \ell$,
\item for each $0\leq j\leq \ell$ the curve $f^{n_j}\circ\sigma\circ\psi$ has $C^r$ size less than $(\eps,\widehat \eps)$.
 \end{itemize}
\end{enumerate}

\begin{lemma}
 The following holds for all $\eps,\widehat{\eps}$ small enough. 
 Let $\cR_k$ be a family of reparametrizations of $\sigma$ over $T'_{n_k}$
which are $(N,\varepsilon,\widehat \varepsilon)$-admissible up to time $n_k$. Then
\begin{equation}\label{e.lower}
\Card(\cR_k)\geq  \rho_k^{2n_k} \exp(\lambda^{min}_k\;n_k)\min \|\sigma'\|.
\end{equation}
\end{lemma}
\begin{proof}
For any reparametrization $\psi$, the curve $f^{n_k}\circ \sigma\circ \psi$ has $C^r$ size less than $(\eps,\widehat\eps)$,
hence 
length  smaller than $1$.
Consequently the set $\sigma(T'_{n_k}\cap \psi[0,1])$ has a length 
smaller than
$\exp(-\lambda^{min}_k\;n_k)$ (by definition of $T_{n_k}$).
Since $\cR_k$ is a family of parameterizations over $T'_{n_k}$,
one gets $\Card(\cR_k)\exp(-\lambda^{min}_k\;n_k)\geq
\sum_{\psi\in \cR_k} |\sigma(T'_{n_k}\cap \psi[0,1])|\geq
|\sigma(T'_{n_k})|\geq \min \|\sigma'\| \cdot |T'_{n_k}|. $
\end{proof}

Next we present some tools for constructing reparametrizations.
To do this 
we need some notation and a couple of more definitions. Balls in $\widehat{M}$ with center $\xi$ and radius $r$ will be denoted by $B(\xi,r)$.  We let $\|D\widehat{f}^n\|_{\text{sup}}:=\sup_{\xi\in\widehat{M}}\|D\widehat{f}^n_\xi\|$, and define
the \emph{asymptotic dilation} 
$$\lambda(\widehat f):=\lim\limits_{n\to\infty}\tfrac 1 n \log \|D\widehat f^n\|_\text{sup}.$$
The \emph{supremum entropy $\overline h(f,\mu)$} of an invariant measure $\mu$
is the essential supremum of the entropies of its ergodic components,  with respect to the natural measure on the space of ergodic components.

%

The following result  is a version of Yomdin's theorem \cite{Yo} adapted to our setting (see~\cite[Theorem 4.13]{BCS}).
We recall that we have fixed $r\geq 2$.

\begin{theorem}[Yomdin]\label{t-yomdin}
There exist $\Upsilon>0$ (depending only on $r$) and $\varepsilon_0>0$ (depending only on $r$ and $f$) such that
for every $\eps,\widehat\eps\in (0,\eps_0)$,
every regular curve $\sigma_0$ with $C^r$ size less than $(\eps,\widehat\eps)$
and every $s\in [0,1]$,
there is a family $\cR$ of reparametrizations of $\sigma_0$ over the set
$$
   T:=\{t\in[0,1]:f(\sigma_0(t))\in B(\sigma_0(s),\eps)\text{ and }
   \hf\circ\widehat \sigma_0(t)\in B(\widehat \sigma_0(s),\widehat \eps)\}
 $$
so that  the curves $f\circ\sigma_0\circ\psi$ have $C^r$ size less than $(\eps,\widehat \eps)$ for each $\psi\in\cR$, and 
$$
\Card(\cR)\leq \Upsilon\|D\widehat f\|_{\rm sup}^{1/(r-1)}.
$$ 
\end{theorem}
The next two propositions, taken from~\cite[Section 5]{BCS}, provide small families of admissible  parametrizations over  certain dynamically defined subsets of $\sigma_0$. (We state them for $f\in C^\infty$, but they only require $f\in C^r$, $r\geq 2$.)

\begin{proposition}[Reparametrization during neutral periods]\label{prop1}
The following holds for some $\gamma_0(r,f,\eta)>0$. 
For all  $0<\gamma\leq \gamma_0(r, f,\eta)$, for all  $N$ large enough,
and for any $\widehat f$-invariant probability measure $\widehat \mu_0$,
there exist $\varepsilon,\widehat \eps>0$ arbitrarily small, $\bar n_0\geq 1$,
and an open set $\widehat U_0$ such that $\widehat \mu_0(\widehat U_0)>1-\gamma^2$,
$\widehat \mu_0(\partial \widehat U_0)=0$ and:

For any regular curve $\sigma_0$ with $C^r$ size less than $(\varepsilon,\widehat\varepsilon)$,
and any $n\geq \bar n_0$, there is a family $\cR$ of reparametrizations of $\sigma_0$ over the set
$$\widehat\sigma_0^{-1}\bigg\{\xi:\; (\xi,\dots,\widehat f^{n-1}(\xi)) \text{ is $\tfrac \eta{10}$-neutral and }
\Card\{\widehat f^j(\xi)\in \widehat U_0,\; 0\leq j<n\}>(1-\gamma)n\bigg\}$$
which is $(N,\varepsilon,\widehat \eps)$-admissible up to time $n$
and has cardinality 
$\Card(\cR)\leq \exp\big[n(\tfrac{\lambda(\widehat f)}{r-1}+\eta)\big]. $
\end{proposition}

In the previous statement $\eps$ has to be chosen much smaller than $\widehat \eps$.
This is the reason why we need to work with different scales on $M$ and $\widehat M$.

\begin{proposition}[Reparametrization during typical orbit segments]\label{prop2}
Let us take $\varepsilon,\widehat\eps>0$ small and $N$ large and $\gamma>0$.
For any $\widehat f$-invariant probability measure $\widehat \mu_1$,
if $n_1$ is large enough, then there exists an open set $\widehat U_1$ satisfying $\widehat \mu_1(\widehat U_1)>1-\gamma^2$,
$\widehat \mu_1(\partial\widehat U_1)=0$ and the following:

For any regular curve $\sigma_0$ with $C^r$ size less than $(\varepsilon,\widehat\varepsilon)$,
there is a family $\cR$ of reparametrizations over $\widehat \sigma_0^{-1}(\widehat U_1)$
which is $(N,\varepsilon,\widehat \eps)$-admissible up to time $n_1$
and with cardinality
$$\Card(\cR)\leq \exp\big[n_1(\overline h(f,\pi_*\widehat \mu_1)+\tfrac{\lambda(\widehat f)}{r-1}+\eta)\big]. $$
\end{proposition}

\subsection{Reparametrizations: construction}
Let $\sigma$ be the curve from section \ref{s.Reduction-to-curve}, and let  $\widehat \mu=(1-\beta) \widehat \mu_0 + \beta \widehat \mu_1$ be the measure from section  \ref{s.subsets}.
Set $\mu_1:=\pi_*(\widehat \mu_1)$.
\begin{proposition}\label{p.upper}
There exist families of reparametrizations $\mathcal{R}_k$ of $\sigma$ over $T'_{n_k}$
which are $(N,\varepsilon,\widehat \varepsilon)$-admissible up to the time $n_k$ for some numbers $N$ (arbitrarily large) and  $\eps,\widehat \varepsilon$ (arbitrarily small), and which  satisfy for all $k$ large
\begin{equation}\label{e.upper}
\Card(\mathcal{R}_k)\leq \exp(\beta h(f,\mu_1)n_k+ c(\eta,\gamma,r)n_k),
\end{equation}
where $c(\eta,\gamma,r)\to 0$ as $\eta,\gamma\to 0$ and $r\to \infty$.
\end{proposition}

\begin{proof}  We follow~\cite[section 7.4]{BCS} closely, using the same division into steps as there.
\bigskip

\noindent
\paragraph{\bf Steps 1--5.} This is essentially the same as in ~\cite[section 7.4]{BCS}, so we only sketch the construction and refer the reader to \cite{BCS} for details:
\begin{enumerate}[--]
\item We fix $r\geq 2$ large and $\eta>0$ small.
\item We decompose  $\widehat \mu_1=\sum a_c\widehat \mu_{1,c}$ such that $\sum a_c=1$
into  $\ell \asymp 1/\eta$ mutually singular  invariant measures satisfying:
almost all the ergodic components of $\mu_{1,c}=\pi_*\widehat \mu_{1,c}$ have their entropy in an interval $[h_c,h_c+\eta)$, for each $c$ (so $\overline{h}(f,\mu_{1,c})\approx h(f,\mu_{1,c})\approx h_c$ up to error  $\eta$).
\item We choose $0<\gamma\ll \eta$ small with $\gamma<\gamma_0(r, f,\eta)$ and $N$ large as in Propositions \ref{prop1}~and~\ref{prop2}.
\item We apply Proposition~\ref{prop1} to $\widehat \mu_0$ and get $\eps,\widehat \eps$ small (small enough so that Proposition~\ref{prop2} applies), $\bar n_0$ and
an open set $\widehat U_0$ with  $\widehat \mu(\partial U_0)=0$.
\item We apply Proposition~\ref{prop2} to each measure $\widehat \mu_{1,c}$
and large distinct integers $n_{1,c}>1/\gamma$. We get open sets $\widehat U_{1,c}$ with $\widehat \mu_{1,c}(\widehat U_{1,c})>1-\gamma^2$ and  $\widehat \mu(\partial \widehat U_{1,c})=0$.
We can reduce them so that for $c\neq c'$ and $0\leq j\leq n_{1,c}$,
we have $\widehat \mu_{1,c}(\widehat U_{1,c'})<\gamma^2$ and
the closures of $\widehat f^j(\widehat U_{1,c})$ and $\widehat U_{1,c'}$ are disjoint.
\item We choose small neighborhoods $\widehat V_0,\widehat V_1$ of the measures $(1-\beta) \widehat \mu_0$  and  $\beta \widehat \mu_1$
such that all measures in a same neighborhood $\widehat V_i$ give the same mass to $\widehat M$ up to an error smaller than $\gamma^2$,
and similarly give the same mass to each set $\widehat U_0$, $\widehat U_{1,1},\dots,\widehat U_{1,\ell}$ up to $\gamma^2$.
\item We choose the neutral parameters $0<\alpha<\tfrac \eta {10}$
and $L> 2\max\{\bar n_0,n_{1,1},\dots,n_{1,\ell}\}/\gamma$ such that for $k$ large enough,
item (b) of Proposition~\ref{p.decomposition} is satisfied for $\widehat V_0,\widehat V_1$.
\end{enumerate}
\bigskip

\noindent
\paragraph{\bf Step 6.} We decompose each non-periodic orbit $(\xi,\dots,\widehat f^{n_k}(\xi))$
into pairwise disjoint subsegments $(\widehat f^{t}(\xi),\dots,\widehat f^{t'-1}(\xi))$, falling into one of the following classes:
\begin{enumerate}[(a)]
\item \emph{Blank segments}: $(\alpha,L)$-neutral, with $\Card\{\widehat f^j(\xi)\in \widehat U_0,\; t\leq j<t'\}>(1-\gamma)(t'-t)$.
\item \emph{Segments with color $c$}: such that $\widehat f^t(\xi)\in \widehat U_{1,c}$ and $t'-t=n_{1,c}$.
\item \emph{Fillers}: orbit segments with length $1$.
\end{enumerate}
Since $n_{1,c}$ are large and distinct, the length of a segment determines its class.

The construction is the same as in \cite[section 7.4]{BCS}, so we only sketch it and refer the reader to \cite{BCS} for details. First, one considers all the $(\alpha,L)$-neutral sub-segments
of $(\xi,\dots,\widehat f^{n_k}(\xi))$ which are maximal for the inclusion.
Those which meet $\widehat U_0$ with a density larger than $1-\gamma$ are tagged  ``blank,"
the others are declared to be made of fillers.
One then considers the complement of the union of all the $(\alpha,L)$-neutral sub-segments:
Scanning from the earliest iterate onwards, one  inductively  selects segments which qualify to be colored segments, and which are disjoint from the segments that have been previously identified.
The iterates that have not been selected  at the end of this process  are declared to be fillers.

\begin{lemma}\label{l.size}
For any $k$ large enough and $\xi\in \widehat \sigma(T'_{n_k})$,
inside the orbit segment $(\xi,\dots,\widehat f^{n_k}(\xi))$,
\begin{enumerate}[(a)]
\item blank segments have total length at least $(1-\beta) n-4\gamma n_k$,
\item segments with color $c$ have total length at most  $\beta a_c n+\gamma n_k$,
\item fillers have total length at most $6\gamma n_k$.
\end{enumerate}
\end{lemma}
\begin{proof}
Item (b) of Proposition~\ref{p.decomposition} and Lemma~\ref{l.sequence}
imply that for $k$ large enough, any point $\xi\in \widehat \sigma(T'_{n_k})$
satisfies $p^{\alpha,L}_k(\xi)\in \widehat V_0$ and $p_k(\xi)-p^{\alpha,L}_k(\xi)\in \widehat V_1$.
Hence the proportion of the orbit segment $(\xi,\dots,\widehat f^{n_k}(\xi))$
and of its $(\alpha,L)$-neutral part spent in each set $\widehat U_0$,
$\widehat U_{1,c}$ is equal to the masses given to these sets by $(1-\beta)\widehat \mu_0$
and $\beta\widehat \mu_1$, up to an error smaller than $\gamma^2$.

We can now repeat the argument in ~\cite[Lemma 7.5]{BCS} verbatim, using
$\widehat \mu_{0}(\widehat U_{1,c})>1-\gamma^2$, $\mu_{1,c}(\widehat U_{1,c})>1-\gamma^2$,
$\mu_{1,c}(\widehat U_{1,c'})<\gamma^2$ ($c\ne c'$), \SY{$\gamma<\gamma_0(r,f,\eta)$}  and $\gamma\ll\eta$.
\end{proof}
\smallskip

\noindent
\paragraph{\bf Step 7.} The non-periodic orbits $(\xi,\dots,\widehat f^{n_k-1}(\xi))$
have been decomposed into segments which begin at iterates $\widehat f^{t_0}(\xi)$,
\dots, $\widehat f^{t_m}(\xi)$. The sequence $\theta=(t_0,\dots,t_m)$ is the \emph{type of the decomposition}.
(We recall that $t_{i+1}-t_i$ determines the class of the segment starting at $t_i$.)

\begin{lemma}\label{l.types}
Let $H(t):=t\log \tfrac 1 t+(1-t)\log\tfrac 1{1-t}.$
For $k$ large, the number of types of decomposition $\theta$ of the orbits $(\xi,\dots,\widehat f^{n_k-1}(\xi))$
with $\xi\in \widehat \sigma(T'_{n_k})$ is bounded by $\exp[H(10\gamma)n_k]$.
\end{lemma}
\begin{proof}[Idea of the proof]
There are at most $6\gamma n_k$ fillers, and segments of other classes have lengths $\geq 1/\gamma$,
hence every type $\theta$ has at most $\lfloor 8\gamma n_k\rfloor+1$ elements.
So the number of possible $\theta$ is bounded by the number of representations of $n_k$ as an ordered sum of  $m\leq \lfloor 8\gamma n_k\rfloor$ positive numbers. The lemma follows from a standard  combinatorial computation,  and  De Moivre's estimate for $n\choose pn$. See \cite[Claim 7.6]{BCS} for details.
\end{proof}

\noindent
\paragraph{\bf [Step 8.} This step is not needed in the present proof (since $n_k$, $T'_{n_k}$ are already built).{\bf ]}
\medskip

\noindent
\paragraph{\bf Step 9.}
We fix a type $\theta$ and set $T_{n_k}^{\theta}=\{t\in T'_{n_k}: 
(\widehat f^i\sigma(t))_{i=0}^{n_k-1}\text{  has a decomposition of type $\theta$}\}$.

\begin{lemma}
There exists a family $\cR^\theta_k$ of reparametrizations of $\sigma$ over the set
$T_{n_k}^{\theta}$ that are $(N,\varepsilon,\widehat \varepsilon)$-admissible up to time $n_k$
with $
\Card(\mathcal R_n^\theta)\leq \exp\bigl(\sum_{i=1}^{m}\kappa_i(\theta)(t_{i}-t_{i-1})\bigr),
$ where
\begin{align*}
\kappa_i(\theta)&:=\begin{cases}
(r-1)^{-1}\lambda(\widehat f)+\eta & \text{if $t_{i}-t_{i-1}\geq L$},\qquad \text{(blank segment)},\\
h_{c}+(r-1)^{-1}\lambda(\widehat f)+2\eta & \text{if $t_{i}-t_{i-1}=n_{1,c}$},\quad \text{(c-colored segment)},\\
\text{some $C$ depending only on $r$ and $f$} & \text{if $t_{i}-t_{i-1}=1$},\qquad \text{(filler)}.
\end{cases}
\end{align*}
\end{lemma}
\begin{proof}[Idea of the proof]
The construction is identical to the step 9 in~\cite{BCS}.
One builds, inductively on $i$, a family of reparametrizations which are $(N,\varepsilon,\widehat \varepsilon)$-admissible up to the time $t_i$.
Given such a reparametrization $\psi$ which is admissible up to the time $t_{i-1}$,
one considers the curve $\sigma_0=f^{t_{i-1}}\circ \sigma\circ \psi$ and applies Proposition~\ref{prop1}, Proposition~\ref{prop2}, or Theorem~\ref{t-yomdin}, depending on the class of the orbit segment between times $t_{i-1}$ and $t_i$ (respectively blank, colored or filler).
One gets a family of reparametrizations $\varphi$ of $\sigma_0$ which are admissible up to time $t_i-t_{i-1}$ over the set $\psi^{-1}(T^{\theta}_{n_k})$.
The set of compositions $\psi\circ\varphi$ defines a family of reparametrizations of $\sigma$ that are admissible up to time $t_i$.
\end{proof}
\smallskip

\noindent
\paragraph{\bf Step 10.}
We estimate the cardinality of a suitable family of reparametrizations $\cR_k$ as in~\cite{BCS}.
For each type $\theta$, the bounds on the total lengths of the different classes of segments (Lemma~\ref{l.size}),
the estimates $\sum a_c h_c\leq h(f,\mu_1)$ and $\gamma\ll \eta$ together give:
$$
\displaystyle
\Card(\mathcal R_n^\theta)\leq \exp\biggl(\beta h(f,\mu_1)n+\frac{\lambda(\widehat f)}{r-1}n+ 4\eta n \biggr).
$$

We can take $\cR_k:=\bigcup_{\theta}\cR^\theta_{k}$.
Its cardinality is at most the above bound multiplied by the number of types, estimated in Lemma~\ref{l.types}. This concludes the proof of Proposition~\ref{p.upper}.
\end{proof}

\subsection{Summary up to this point and conclusion of the proof}
Fix a regular $C^\infty$ curve $\sigma:[0,1]\to M$ so that $\limsup\frac{1}{n}\log \|Df^n|_{\sigma}(\sigma(s))\|>0$ for a set of positive Lebesgue measure of parameters $s\in [0,1]$ (Lemma \ref{l.curve}).

Let $\overline{\lambda}:=\mathrm{ess\,sup}_{s\in [0,1]}\limsup \frac{1}{n}\log \|Df^n|_{\sigma}(\sigma(s))\|>0$,  and choose sequences $\rho_k\uparrow 1$, $\lambda^{min}_k\uparrow\overline{\lambda}$ and $\lambda^{max}_k\downarrow\overline{\lambda}$.
In sections  \ref{s.Reduction-to-curve}--\ref{s.subsets} we constructed  $n_k\uparrow\infty$, and a sequence of sets
\begin{multline*}
T_{n_k}'\subset \{s\in [0,1]:\lambda_{k}^{min}\leq \frac{1}{n_k}\log\|Df^{n_k}|_{\sigma}(\sigma(s))\|\leq \lambda_{k}^{max},\\
\sigma(s)\text{ is not $f$-periodic and }\forall 0\leq i\leq n_k, \,
\|Df^i|_{\sigma}(f^{n_k-i}\sigma(s))\|\geq 1\}
\end{multline*}
with Lebesgue measure $|T_{n_k}'|\geq \rho_k^{2k}$.

We then took points  $\xi_{n_k}\in \widehat{\sigma}(T_{n_k}')$, and replaced $n_k$ by a further subsequence (which abusing notation we also called $n_k$), for which  the $\widehat{f}$-empirical measures $p_k(\xi_{n_k})$ satisfy
$$
p_k(\xi_{n_k})\xrightarrow[k\to\infty]{\text{weak-}\ast}\widehat{\mu}=(1-\beta)\widehat{\mu}_0+\beta\widehat{\mu}_1,
$$
with $\widehat{\mu}_0$, $\widehat{\mu}_1$  as in Proposition \ref{p.decomposition}. We saw in Corollary \ref{c.decomposition} that $\beta>0$ and that the projection of $\widehat{\mu}_1$ to $M$ is an $f$-invariant measure such that
\begin{equation}\label{e.lambda-plus}
\beta\lambda^+(\mu_1)=\overline{\lambda}.
\end{equation}

We then analyzed the Yomdin-theoretic $C^r$ complexity of $\sigma|_{T_{n_k}'}$, by estimating the size of the family $\mathcal R_k$ of admissible reparametrizations needed to cover $\sigma(T_{n_k}')$. Using tools developed in \cite{BCS}, we were able to show the existence of families $\mathcal R_k$ such that for $k\gg 1$,
\begin{equation}\label{e.double}
\rho_k^{2n_k}\exp(\lambda_k^{min} n_k)\min\|\sigma'\|\leq \mathrm{Card}(\mathcal R_k)\leq \exp(\beta h(f,\mu_1)n_k+c(\eta,\gamma,r)n_k),
\end{equation}
with $\eta, \gamma$ arbitrarily small, $r$ arbitrarily large, and $c(\eta,\gamma,r)\to 0$ as $\gamma,\eta\to 0, r\to\infty$, see \eqref{e.lower},\eqref{e.upper}.
Since  $\lambda_k^{min}\to \overline{\lambda}$ and $\rho_k\to 1$, \eqref{e.double} implies that
$\overline{\lambda}\leq \beta h(f,\mu_1).$

Combining this with \eqref{e.lambda-plus} and with Ruelle's inequality, we obtain:
$$\beta \lambda^+(\mu_1) =\overline \lambda\leq \beta h(f,\mu_1)\leq\beta \lambda^+(\mu_1) \qquad\text{ with }\overline\lambda>0.$$
Hence $\lambda^+(\mu_1) =h(f,\mu_1)>0$. Note that almost every ergodic component $\mu_1'$ of $\mu_1$ satisfies $\lambda^+(\mu_1')=h(f,\mu_1')$ by Ruelle's inequality and $\lambda^+(\mu_1')>0$ by Proposition~\ref{p.decomposition} (d), hence is an SRB measure.
\qed

\bigskip
\bigskip

\hspace{-1.5cm}
\begin{tabular}{l l l l l}
\emph{J\'er\^ome Buzzi}
& &\emph{Sylvain Crovisier}
& &\emph{Omri Sarig}\\
Laboratoire de Math\'ematiques
&& Laboratoire de Math\'ematiques
&& Faculty of Mathematics\\
 d'Orsay, CNRS - UMR 8628
&&  d'Orsay, CNRS - UMR 8628
&&  and Computer Science\\
Universit\'e Paris-Saclay
&&  Universit\'e Paris-Saclay
&& The Weizmann Institute of Science\\
Orsay 91405, France
&& Orsay 91405, France
&& Rehovot, 7610001,  Israel\\
\end{tabular}

\end{document}